\documentclass[runningheads]{llncs}

\usepackage{graphicx} 
\usepackage{amssymb,amsfonts}
\usepackage[english]{babel}
\usepackage{hyperref}
\usepackage{url}
\usepackage{bussproofs}
\usepackage{proof}
\usepackage{cancel}
\usepackage{eqnarray,amsmath}

\begin{document}
\title{Two-sorted Modal Logic for Formal and Rough Concepts\thanks{This work is partially supported by National Science and Technology Council (NSTC) of Taiwan under Grant  No. 110-2221-E-001-022-MY3}
}
%
%
\author{Prosenjit Howlader \and
 Churn-Jung Liau
}
\authorrunning{P. Howlader et al.}
%
\institute{Institute of Information Science, Academia Sinica, Taipei, 115, Taiwan\\
\email{prosen@mail.iis.sinica.edu.tw}\\
\email{liaucj@iis.sinica.edu.tw}}
\maketitle              
\begin{abstract}
In this paper, we propose two-sorted modal logics for the representation and reasoning of concepts arising from rough set theory (RST) and formal concept analysis (FCA). These logics are interpreted in two-sorted bidirectional frames, which are essentially formal contexts with converse relations. On one hand, the logic $\textbf{KB}$ contains ordinary necessity and possibility modalities and can represent rough set-based concepts. On the other hand, the logic $\textbf{KF}$ has window modality  that can represent formal concepts.
We study the relationship between  \textbf{KB} and \textbf{KF} by proving a correspondence theorem.
It is then shown that, using the formulae with modal operators in \textbf{KB} and \textbf{KF}, we can capture formal concepts based on RST and FCA and their lattice structures.

\end{abstract}
\section{Introduction}
Rough set theory (RST) \cite{pawlak1982rough} and formal concept analysis (FCA) \cite{FCA} are both well-established areas of study with a variety of applications in fields like knowledge representation and data analysis. There has been a great deal of research on the intersections of RST and FCA over the years, including those by Kent \cite{RCA}, Saquer et al \cite{saquer2001concept}, Hu et al \cite{hul}, D\"{u}ntsch and Gediga \cite{duntsch2002modal}, Yao \cite{yao2004concept}, Yao et al \cite{RSAFCAyao}, Meschke \cite{aclmc} , and Ganter et al \cite{FCAARDT}.

Central notions in FCA are formal contexts and their associated concept lattices.  A formal context (or simply context) is a triple  $\mathbb{K}:=(G, M, I)$ where $I\subseteq G\times M$. A given context induces two maps $+: (\mathcal{P}(G), \subseteq)\rightarrow (\mathcal{P}(M), \supseteq) $ and $-: (\mathcal{P}(M), \supseteq)\rightarrow (\mathcal{P}(G), \subseteq)$, where for all $A\in \mathcal{P}(G)$ and $B\in \mathcal{P}(M)$:
\[A ^{+}= \{ m \in  M \mid \mbox{ for all }  g \in  A ~~  g I m \},\]
\[B ^{-}= \{ g \in  G \mid \mbox{ for all }  m \in  B ~~  g I m \}.\]





A pair of set $(A, B)$ is called a {\it formal  concept} (or simply concept) if $A^{+}=B$ and $A=B^{-}$. The set $\mathcal{FC}$ of all concepts forms a complete lattice and is called a {\it concept lattice}.  

On the other hand, the basic construct of the original RST is the {\it Pawlakian approximation space} $(W, E)$, where $W$ is the universe and $E$ is an equivalence relation on $W$. Then, by applying notions of modal logic to RST, Yao et al~\cite{yao1996generalization} proposed generalised approximation space $(W, E)$ with $E$ being any binary relation on $W$. In addition, they also suggested to use a binary relation between two universes of discourse, containing
objects and properties respectively, as another generalised formulation of approximation spaces.  The rough set model over two universes is thus a formal context in FCA. D\"{u}ntsch et al. \cite{duntsch2002modal} defined sufficiency, dual sufficiency, possibility and necessity operators based on a rough set model over two universes, where necessity and possibility operators are, in fact, rough set approximation operators. Based on these operators, D\"{u}ntsch et al. \cite{duntsch2002modal} and Yao \cite{yao2004concept}  introduced property oriented concepts and object oriented concepts respectively.

For a context  $\mathbb{K}:=(G,M,I)$, $I(x):= \{y\in M: xI y\}$ and $I^{-1}(y):=\{x\in G:xIy\}$ are the $I$-neighborhood and $I^{-1}$-neighbourhood of $x$ and $y$ respectively. For $A\subseteq G$, and $B\subseteq M$, the pairs of dual approximation operators are defined as:

$B_{I}^{\lozenge^{-1}}:=\{x\in G:I(x)\cap B\neq\emptyset\},~~~~ ~~~~~B_{I}^{\square^{-1}}:=\{x\in G:I(x)\subseteq B\}$.

$A_{I^{-1}}^{\lozenge}:=\{y\in M:I^{-1}(y)\cap A\neq\emptyset\},~~~ A_{I^{-1}}^{\square}:=\{y\in M:I^{-1}(y)\subseteq A\}$.

If there is no confusion about the relation involved, we shall omit the subscript and denote $B_{I}^{\lozenge^{-1}}$ by  $B^{\lozenge^{-1}}$, $B_{I}^{\square^{-1}}$ by $B^{\square^{-1}}$ and  similarly for the case of $A$. A pair $(A,B)$ is a  {\it property oriented concept}  of  $\mathbb{K}$ iff $ A^{\lozenge}=B$ and $B^{\square^{-1}}=A$; and it  is an {\it object oriented concept}  of $\mathbb{K}$ iff $ A^{\square}=B$ and $B^{\lozenge^{-1}}=A$. As in the case of FCA, the set $\mathcal{OC}$ of all object oriented concepts and the set $\mathcal{PC}$ of all property oriented concepts form complete lattices,  which are called {\it object oriented concept lattice} and {\it property oriented concept lattice} respectively.

For any concept $(A, B)$, the set $A$ is called its {\it extent } and $B$ is called its {\it intent}. For concept lattices ${\mathcal X}={\mathcal FC},{\mathcal PC},{\mathcal OC}$, the set of all extents and intents of $\mathcal X$  are denoted  by  $\mathcal{X}_{ext}$ and $\mathcal{X}_{int}$, respectively.

\begin{proposition}
    {\rm For a context $\mathbb{K}:=(G, M, I)$, the following holds.
    \begin{itemize}
        \item[(a)] $\mathcal{FC}_{ext}=\{ A\subseteq G\mid A^{+-}=A\}$ and $\mathcal{FC}_{int}=\{B\subseteq M\mid B^{-+}=B\}$.
        \item[(b)] $\mathcal{PC}_{ext}=\{ A\subseteq G\mid A^{\lozenge\square^{-1}}=A\}$ and $\mathcal{PC}_{int}=\{B\subseteq M\mid B^{\square^{-1}\lozenge}=B\}$.
        \item[(c)] $\mathcal{OC}_{ext}=\{A\subseteq G\mid A^{\square\lozenge^{-1}}=A\}$ and $\mathcal{OC}_{int}=\{ B\subseteq M\mid B^{\lozenge^{-1}\square}=B\}$.
    \end{itemize}}
\end{proposition}

It can be shown that the sets $\mathcal{FC}_{ext}, \mathcal{PC}_{ext}$ and $\mathcal{OC}_{ext}$ form complete lattices and are isomorphic to the corresponding concept  lattices. Analogously, the sets $\mathcal{FC}_{int}, \mathcal{PC}_{int}$ and $\mathcal{OC}_{int}$ form complete lattices and are dually isomorphic to the corresponding concept lattices. Therefore, a concept can be identify with its extent or intent. The relationship between these two kinds of rough concept lattices and concept lattices of FCA are investigated in \cite{yao2004comparative}. 
In particular, the following theorem is proved.
\begin{theorem}
    {\rm \cite{yao2004comparative} For a context $\mathbb{K}=(G, M, I)$ and the complemented context $\mathbb{K}^{c}=(G, M, I^{c})$, the following holds.
    \begin{itemize}
        \item[(a)] The concept lattice of $\mathbb{K}$ is isomorphic to the property oriented concept lattice of $\mathbb{K}^{c}$.
        \item[(b)] The property oriented concept lattice of $\mathbb{K}$ is dually  isomorphic to the object oriented concept lattice of $\mathbb{K}$.
        \item[(c)] The concept lattice of $\mathbb{K}$ is dually isomorphic to the object oriented concept lattice of $\mathbb{K}^{c}$.
    \end{itemize}}
\end{theorem}

In addition, to deal with the negation of  concept,  the notions of {\it semiconcepts} and {\it protoconcepts} are introduced in \cite{wille}. Algebraic studies of these notions led to the definition of double Boolean algebras and pure double Boolean algebras \cite{wille}. These structures have been investigated by many authors \cite{wille,vormbrock2005semiconcept,BALBIANI2012260,MR4566932}. There is also study of logic corresponding to these algebraic structures \cite{howlader2021dbalogic,HOWLADER2023115}. 

The operators used in formal and rough concepts correspond to modalities used in modal logic  \cite{Gargov1987,blackburn2002moda}. In particular, the operator used in FCA is the window modality (sufficiency operator) \cite{Gargov1987} and those used in RST are box (necessity operator) and diamond (possibility operator) \cite{blackburn2002moda}. Furthermore, a context is a two-sorted structure consisting of a set of objects and a set of properties. Considering these facts, our goal in this work is to formulate two-sorted modal logics that are sound and complete with respect to the class of all contexts and can represent all the three kinds of concepts and their lattices.

To achieve the goal, we first introduce the notion of {\it two-sorted bidirectional frame}, which is simply a formal context extended with the converse of the binary relation. 
Then, we propose two-sorted modal logics \textbf{KB} and \textbf{KF} as representation formalism for rough and formal concepts respectively, and two-sorted bidirectional frames serve as semantic models of the logics. We also prove the soundness and completeness of the proposed logics with respect to the semantic models.

Next, we will review basic definitions and main results of general many-sorted polyadic modal logic. Then, in Section \ref{KB}, we define the logic \textbf{KB} and  characterize the pairs of formula that represent property and object oriented concepts of context. The logic \textbf{KF} and formal concept are discussed in Section \ref{KF}. We revisit the three concept lattices and their relations in terms of logic in Section \ref{lattice}. Finally, we summarize the paper and indicate directions of future work in Section \ref{conclusion}.

\subsection{Many-sorted polyadic modal logic} \label{mspml}
The many-sorted polyadic modal logic is introduced in \cite{mspml}. The alphabet of the logic consists of a many-sorted signature $(S, \Sigma)$, where $S$ is the collection of sorts and $\Sigma$ is the set of modalities, and an {\it $S$-indexed} family $P:=\{P_{s}\}_{s\in S}$ of propositional variables, where $P_{s}\neq \emptyset$ and $P_{s}\cap P_{t}= \emptyset$ for distinct $s, t\in S$. Each modality $\sigma\in\Sigma$ is associated with an arity $s_{1}s_{2}\ldots s_{n}\rightarrow s$. For any $n\in \mathbb{N}$, we denote  $\Sigma_{s_{1}s_{2}\ldots s_{n}s}=\{\sigma\in \Sigma\mid \sigma:s_{1}s_{2}\ldots s_{n}\rightarrow s\}$

For an $(S, \Sigma)$-modal language $\mathcal{ML}_{S}$, the set of formulas  is an $S$-index family $Fm_{S}:=\{Fm_{s}\mid s\in S\}$, defined inductively for each $s\in S$ by
\[\phi_s::= p_s\;\mid\;\neg\phi_s\;\mid\;\phi_s\wedge\phi_s\;\mid\;\sigma(\phi_{s_1}\ldots \phi_{s_n})\;\mid\;\sigma^{\square}(\phi_{s_1}\ldots \phi_{s_n}),\]
where $p_s\in P_s$ and $\sigma\in \Sigma_{s_{1}s_{2}\ldots s_{n}s}$.

A {\it many-sorted relational frame} is a pair  $\mathfrak{F}:=(\{W_{s}\}_{s\in S}, \{R_{\sigma}\}_{\sigma\in \Sigma})$ where $W_{s}\neq \emptyset$, $W_{s_{i}}\cap W_{s_{j}}=\emptyset$ for  $s, s_{i}\not= s_{j}\in S$ and $R_{\sigma}\subseteq W_{s}\times W_{s_{1}}\ldots \times W_{s_{n}}$  if  $\sigma \in \Sigma_{s_{1}s_{2}\ldots s}$. The class of all  many-sorted relational frames is denoted as $\mathbb{SRF}$. A {\it valuation} $v$ is an $S$-indexed family of  maps $\{v_{s}\}_{s\in S}$, where $v_{s}: P_{s}\rightarrow \mathcal{P}(W_{s})$. A  many-sorted model $\mathfrak{M}:=(\mathfrak{F}, v)$ consists of a many-sorted frame $\mathfrak{F}$ and a valuation $v$. The satisfaction of a  formula in a model $\mathfrak{M}$ is defined inductively as follows.

\begin{definition}
\label{satisfiction}
{\rm Let $\mathfrak{M}:=(\{W_{s}\}_{s\in S}, \{R_{\sigma}\}_{\sigma\in \Sigma}, v)$ be a  many-sorted model,  $w\in W_{s}$ and $\phi \in Fm_{s}$  for $s\in S$. We define $\mathfrak{M}, w\models_{s} \phi$  by induction over $\phi$ as follows:
\begin{enumerate}
    \item $\mathfrak{M},w\models_{s} p$ iff $w\in v_{s}(p)$
    \item $\mathfrak{M},w\models_{s}\neg\phi$ iff $\mathfrak{M},w\not\models_{s}\phi$
    \item $\mathfrak{M},w\models_{s} \phi_{1}\wedge \phi_{2}$ iff $\mathfrak{M},w\models_{s}\phi_{1}$ and $\mathfrak{M},w\models_{s}\phi_{2}$
    \item If $\sigma\in \Sigma_{s_{1}s_{2}\ldots s}$, then $\mathfrak{M},w\models_{s} \sigma (\phi_{1}, \phi_{2}\ldots \phi_{n})$ iff there is $(w_{1}, w_{2}\ldots w_{n})\in W_{s_{1}}\times W_{s_{2}}\ldots W_{s_{n}}$ such that $(w,w_{1}, w_{2}\ldots w_{n})\in R_{\sigma}$ and $\mathfrak{M},w_{i}\models_{s_{i}} \phi_{i}$ for $i\in \{1, 2\ldots n\}$
\end{enumerate}

    }
\end{definition}

\begin{definition}{\rm \cite{mspml} Let $\mathfrak{M}$ be an $(S,\Sigma)$-model. Then, for a set $\Phi_{s}$ of formula, $\mathfrak{M},w\models_{s}\Phi_{s}$ if $\mathfrak{M},w\models_{s}\phi$ for all $\phi\in\Phi_{s}$.

Let $\mathcal{C}$ be a class of models. Then, for a set $\Phi_s\cup\{\phi\}\subseteq Fm_{s}$, $\phi$ is a local semantic consequence of $\Phi_{s}$ over $\mathcal{C}$ and denoted as $\Phi_{s}\models^{\mathcal{C}}_{s}\phi$  if $\mathfrak{M}, w\models_{s}\Phi_{s}$ implies $\mathfrak{M}, w\models_{s}\phi$ for all models $\mathfrak{M}\in\mathcal{C}$. If $\mathcal{C}$ is the class of all models, we omit the superscript and denote it as $\Phi_{s}\models_{s}\phi$.

If $\Phi_{s}$ is empty, we say $\phi$ is valid in  $\mathcal{C}$ and denoted it as $\mathcal{C}\models_{s}\phi$. When $\mathcal{C}$ is the class of all models based on a given frame $\mathfrak{F}$, we also denote it by $\mathfrak{F}\models_{s}\phi$.}
\end{definition}

To characterize the local semantic consequence, the modal system $\mathbf{K}_{(S, \Sigma)}:=\{\mathbf{K}_s\}_{s\in S}$ is proposed in \cite{mspml}, where $\mathbf{K}_s$ is the axiomatic system in Figure~\ref{fig1} in which $\sigma\in\Sigma_{{s_1}\ldots s_n,s}$:
\begin{figure}[htp]
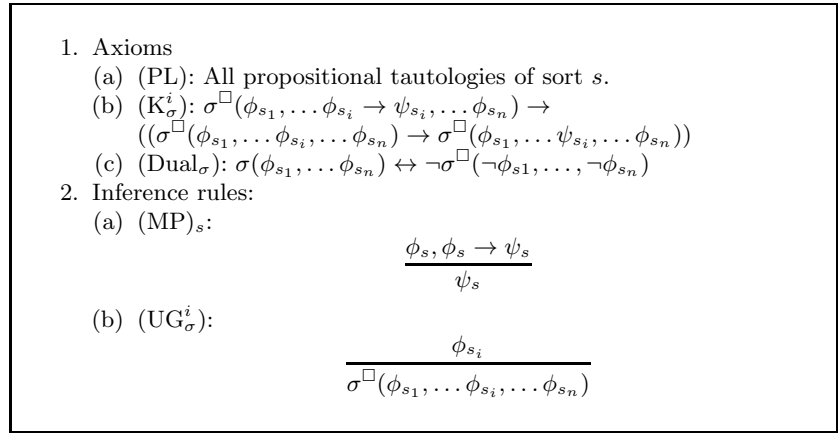
\centering
	\framebox[110mm] {\parbox{100mm}{\begin{enumerate}
\item Axioms
\begin{enumerate}
    \item (PL): All propositional tautologies of sort $s$.
    \item (K$_{\sigma}^{i}$): $\sigma^{\square}(\phi_{s_1}, \ldots \phi_{s_i}\rightarrow \psi_{s_i}, \ldots \phi_{s_n})\rightarrow\\( (\sigma^{\square}(\phi_{s_1}, \ldots \phi_{s_i}, \ldots \phi_{s_n}) \rightarrow \sigma^{\square}(\phi_{s_1}, \ldots \psi_{s_i}, \ldots \phi_{s_n})) $
    \item (Dual$_{\sigma}$): $\sigma (\phi_{s_1}, \ldots \phi_{s_n})\leftrightarrow\neg \sigma^{\square}(\neg\phi_{s1}, \ldots, \neg\phi_{s_n})$
\end{enumerate}
\item  Inference rules:
\begin{enumerate}
    \item (MP)$_{s}$: \[\infer{\psi_s}{\phi_s, \phi_s\rightarrow \psi_s}\]
    \item (UG$^{i}_{\sigma})$: \[\infer{\sigma^{\square}(\phi_{s_1}, \ldots \phi_{s_i}, \ldots \phi_{s_n})}{\phi_{s_i}}\]
\end{enumerate}
\end{enumerate}
}} \caption{The axiomatic system $\mathbf{K}_s$}\label{fig1}
\end{figure}

When the signature is clear from the context, the subscripts may
be omitted and we simply write the system as $\mathbf{K}$.
\begin{definition}
    {\rm \cite{mspml} Let $\Lambda\subseteq Fm_S$ be an $S$-sorted  set of formulas. The normal modal logic defined by  $\Lambda$ is $\mathbf{K}\Lambda:=\{\mathbf{K}\Lambda_s\}_{s\in S}$ where $\mathbf{K}\Lambda_s:=\mathbf{K}_s\cup\{\lambda^{\prime}\in Fm_s \mid \lambda^{\prime}$ is obtained by uniform substitution applied to a formula $\lambda\in \Lambda_s\}$.}
\end{definition}

\begin{definition}{\rm \cite{mspml}
A sequence of formulas $\phi_{1}, \phi_{2}, \ldots\phi_{n}$ is called a $\mathbf{K}\Lambda$-proof for the formula $\phi$  if  $\phi_{n}=\phi$ and  $\phi_{i}$ is in $\mathbf{K}\Lambda_{s_{i}}$ or inferred from $\phi_{1}, \ldots, \phi_{i-1}$ using modus pones and universal generalization.  If $\phi$ has a proof in $\mathbf{K}\Lambda$,  we say that $\phi$ is a theorem and write $\vdash_{s}^{\mathbf{K}\Lambda}\phi$.
Let $\Phi\cup\{\phi\}\subseteq Fm_s$ be a set of formulas. Then, we say that $\phi$ is provable form $\Phi$, denoted by $\Phi\vdash_{s}^{\mathbf{K}\Lambda}\phi$, if there exist $\phi_1,\ldots,\phi_n\in\Phi$ such that $\vdash_{s}^{\mathbf{K}\Lambda}(\phi_1\wedge\ldots\wedge\phi_n)\rightarrow\phi$. In addition, the set $\Phi$ is $\mathbf{K}\Lambda$-inconsistent  if $\bot$ is provable from it, otherwise it is $\mathbf{K}\Lambda$-consistent.}
\end{definition}
\begin{proposition}
\label{suffconstrongcomp}
   {\rm  \cite{mspml} $\mathbf{K}\Lambda$  is strongly complete with respect to a class of models $\mathcal{C}$ if and only if any  consistent set $\Gamma$ of formulas is satisfied in some model from  $\mathcal{C}$.}
\end{proposition}
\begin{definition}
    \label{canonicalmodel}
    {\rm \cite{mspml} The canonical model is \[\mathfrak{M}^{\mathbf{K}\Lambda}:=(\{W_{s}^{\mathbf{K}\Lambda}\}_{s\in S}, \{R_{\sigma}^{\mathbf{K}\Lambda}\}_{\sigma\in\Sigma}, V^{\mathbf{K}\Lambda})\] where
    \begin{itemize}
        \item[(a)] for any $s\in S$,  $W_{s}^{\mathbf{K}\Lambda}=\{\Phi\subseteq Fm_{s}~\mid~ \Phi ~\mbox{is maximally} ~\mathbf{K}\Lambda\mbox{-consistent}\}$,
        \item[(b)] for any $\sigma\in \Sigma_{s_{1}\ldots s_{n},s}$, $w\in W_{s}^{\mathbf{K}\Lambda}, u_{1}\in W_{s_{1}}^{\mathbf{K}\Lambda},\ldots u_{n}\in W_{s_{n}}^{\mathbf{K}\Lambda} $, $R_{\sigma}^{\mathbf{K}\Lambda}wu_{1}\ldots u_{n}$ if and only if $(\psi_{1}, \ldots,\psi_{n})\in u_{1}\times u_{2}\times\ldots \times u_{n}$ implies that $\sigma(\psi_{1}, \ldots, \psi_{n})\in w$.
        \item[(c)] $V^{\mathbf{K}\Lambda}=\{V^{\mathbf{K}\Lambda}_{s}\}$ is the valuation defined by $V_{s}^{\mathbf{K}\Lambda}(p)=\{w\in W_{s}^{\mathbf{K}\Lambda}~\mid~ p\in w\}$ for any $s\in S$ and $p\in P_{s}$.
    \end{itemize}}
\end{definition}

\begin{lemma}
\label{truthlemma}
    {\rm \cite{mspml} If $s\in S$, $\phi\in Fm_{s}$, $\sigma \in \Sigma_{s_{1}\ldots s_{n}, s}$ and $w\in W_{s}^{\mathbf{K}\Lambda}$ then the following hold:
    \begin{itemize}
        \item[(a)] $R_{\sigma}^{\mathbf{K}\Lambda}wu_{1}\ldots u_{n}$ if and only if for any formulas $\psi_{1}, \ldots,\psi_{n}$, $\sigma^{\square}(\psi_{1}, \ldots,\psi_{n})\in w$ implies $\psi_{i}\in u_{i}$ for some $i\in \{1, 2,\ldots, n\}$.
        \item[(b)] If $\sigma(\psi_{1}, \ldots,\psi_{n})\in w$ then for any $i\in \{1,2\ldots , n\}$ there is $u_{i}\in W_{s_{i}}^{\mathbf{K}\Lambda}$ such that $\psi_{1}\in u_{1},\ldots, \psi_{n}\in u_{n}$ and $R_{\sigma}^{\mathbf{K}\Lambda}wu_{1}\ldots u_{n}$.
        \item[(c)] $\mathfrak{M}^{\mathbf{K}\Lambda}, w\models_{s}\phi$ if and only if $\phi\in w$.
    \end{itemize}}
\end{lemma}

 \begin{proposition}
 \label{compcanonic}
     {\rm \cite{mspml} If $\Phi_s$ is a $\mathbf{K}\Lambda$-consistent set of formulas then it is satisfied in  the canonical model. }
 \end{proposition}

These results implies the soundness and completeness of $\mathbf{K}$ directly.
\begin{theorem}\label{mainsoundcomplete}
 {\rm  $\mathbf{K}$ is sound and strongly complete with respect to the class of all $(S,\Sigma)$-models, that is, for any $s\in S$, $\phi\in Fm_{s}$  and $\Phi_{s}\subseteq Fm_{s}$, $\Phi_{s}\vdash_{s}^{\mathbf{K}}\phi$ if and only if $\Phi_{s}\models_{s}\phi$. }
\end{theorem}

\section{Two-sorted modal logic and concept lattices}
In this section, we present the logics \textbf{KB} and \textbf{KF} and discuss their relationship with rough and formal concepts.

\subsection{Two-sorted modal logic and concept lattices in rough set theory}\label{KB}
Let us consider a special kind of two-sorted signature  $(\{s_{1}, s_{2}\}, \Sigma)$ where $\Sigma=\Sigma_{1}\uplus\Sigma_{2}$ is the direct sum of two sets of unary modalities such that $\Sigma_{1}=\Sigma_{s_{1}s_{2}}$ and $\Sigma_2=\Sigma_{s_{2}s_{1}}=\Sigma_{1}^{-1}:=\{\sigma^{-1}: \sigma\in \Sigma_{1}\}$. We say that the signature is bidirectional. Modal languages built over bidirectional signatures are interpreted in bidirectional frames.


\begin{definition}
 {\rm For the signature above, a two-sorted bidirectional frame is a quadruple :
 \[\mathfrak{F}_{2}:=(W_{1}, W_{2},  \{R_{\sigma}\}_{\sigma\in \Sigma_{1}}, \{R_{\sigma^{-1}}\}_{\sigma\in \Sigma_{1}})\] where $W_{1}, W_{2}$ are non-empty disjoint sets and $R_{\sigma}\subseteq W_2\times W_1$, $R_{\sigma^{-1}}$ is the converse of $R_{\sigma}$. The class of all two-sorted bidirectional frame is denoted as $\mathbb{BSFR}_{2}$.}
\end{definition}

The logic system $\mathbf{KB}$ for two-sorted bidirectional frames is define as $\mathbf{K}\Lambda$ where $\Lambda$ consists of the following axioms:
\[~\mbox{(B)}~ p\rightarrow (\sigma^{-1})^{\square}\sigma p~\mbox{and}~ q\rightarrow \sigma^{\square}\sigma^{-1} q~\mbox{where}~ p\in P_{s_{1}}~\mbox{and}~q\in P_{s_{2}}.\]
\begin{theorem}
\label{pmbdls}
{\rm $\mathbf{KB}$ is sound with respect to class $\mathbb{BSFR}_{2}$ of all two-sorted bidirectional frame.}
\end{theorem}
\begin{proof}
    The proof 
    is straightforward. Here we give the proof for the axiom $p\rightarrow (\sigma^{-1})^{\square}\sigma p$. Let $\mathfrak{M}$ be a model based on the frame $\mathfrak{F}_{2}$ defined above and $\mathfrak{M}, w_{1}\models_{s_{1}} p$ for some $w_{1}\in W_{1}$.  Now, for any $w_2\in W_{2}$ such that $R_{\sigma^{-1}}w_{1}w_2$, we have  $\mathfrak{M},w_2\models_{s_2}\sigma p$  because $R_{\sigma}w_2w_{1}$ follows from the converse of relation. This leads to $\mathfrak{M}, w_{1}\models_{s_{1}}(\sigma^{-1})^{\square}\sigma p$ immediately.
\end{proof}

The completeness theorem is proved using the canonical model of \textbf{KB}, which is an instance of that constructed in Definition \ref{canonicalmodel}. Hence,
\[\mathfrak{M}^{\textbf{KB}}:=(\{W_{s_{1}}^{\textbf{KB}}, W_{s_{2}}^{\textbf{KB}}\}, \{R_{\sigma}^{\textbf{KB}}, R_{\sigma^{-1}}^{\textbf{KB}}\\\}_{\sigma\in\Sigma}, V^{\textbf{KB}})\] It is easy to see that the model satisfies the following properties for $x\in W_{s_{1}}^{\textbf{KB}}$ and $y\in W_{s_{2}}^{\textbf{KB}}$:
\begin{itemize}
 \item[(a)] $R_{\sigma}^{\textbf{KB}}yx$ iff $\phi\in x$ implies that $\sigma\phi\in y$ for any $\phi\in Fm_{s_1}$.
\item[(b)] $R_{\sigma^{-1}}^{\textbf{KB}}xy$ iff $\phi\in y$ implies that $\sigma^{-1}\phi\in x$ for any $\phi\in Fm_{s_2}$.
\end{itemize}

\begin{theorem}\label{pmbdlc}
{\rm $\mathbf{KB}$ is strongly complete with respect to class of all two-sorted bidirectional models, that is for any $s\in \{s_{1}, s_{2}\}$, $\phi\in Fm_{s}$  and $\Phi_{s}\subseteq Fm_{s}$,  $\Phi_{s}\models^{\mathbb{BSFR}_{2}} _{s}\phi$ implies that $\Phi_{s}\vdash_{s}^{\mathbf{KB}}\phi$. }
\end{theorem}

\begin{proof}
    It is sufficient to show that the canonical model is a bidirectional frame. Then, the result follows from Propositions~\ref{suffconstrongcomp} and \ref{compcanonic}. Let $x\in W_{s_{1}}^{\mathbf{KB}}$  and $y\in W_{s_{2}}^{\mathbf{KB}}$ and assume $(y, x)\in R_{\sigma}^{\mathbf{KB}}$. Then, for any $\phi\in y$, we have $\sigma^{\square}\sigma^{-1}\phi\in y$ by axiom (B), which in turns implies $\sigma^{-1}\phi\in x$ by Lemma \ref{truthlemma}. Hence, $(x,y)\in R_{\sigma^{-1}}^{\mathbf{KB}}$ by property (b) of the canonical model. Analogously, we can show that $(x,y)\in R_{\sigma^{-1}}^{\mathbf{KB}}$ implies $(y, x)\in R_{\sigma}^{\mathbf{KB}}$. That is, $R_{\sigma^{-1}}^{\mathbf{KB}}$ is indeed the converse of $R_{\sigma}^{\mathbf{KB}}$.
\end{proof}

To represent rough concepts, we consider a particular bidirectional signature $(\{s_{1}, s_{2}\}, \{\lozenge, \lozenge^{-1}\})$ (i.e. the signature that $\Sigma_1$ is a singleton containing the modality $\lozenge$). As usual, we denote the dual modalities of $\lozenge$ and $\lozenge^{-1}$ by $\square$ and $\square^{-1}$ respectively. Let $\mathcal{SF}_{2}$ denote
the class of all bidirectional frames over the signature and let $\mathcal{K}$ be the set of all contexts. Then, there is a bijective correspondence between $\mathcal{K}$ and $\mathcal{SF}_{2}$ given by $(G, M, I)\mapsto (G, M, I^{-1}, I)$. Note that $I^{-1}$ and $I$ respectively correspond to modalities $\lozenge$ and $\lozenge^{-1}$ under the mapping.  We use $Fm(\textbf{RS}):=\{Fm(\textbf{RS})_{s_{1}}, Fm(\textbf{RS})_{s_{2}}\}$  and $\textbf{KB}_{2}$ to denote the indexed family of formulas and its logic system over the particular signature respectively. By Theorems \ref{pmbdls} and \ref{pmbdlc},  $\mathbf{KB}_{2}$ is sound and complete with respect to the class $\mathcal{SF}_{2}$ and hence $\mathcal{K}$.

Let us denote the truth set of a formula $\phi\in Fm(\textbf{RS})_{s_{i}} (i=1, 2)$ in a model $\mathfrak{M}$ by $[[\phi]]_\mathfrak{M}:=\{w\in W_{i}\mid \mathfrak{M}, w\models_{s_{i}} \phi\}$. We usually omit the subscript and simply write $[[\phi]]$.

\begin{proposition}
  {\rm  Let $\mathbb{K}:=(G,M,I)$ be a context and $\mathfrak{M}:=(G,M,I^{-1},I,v)$ be a model based on its corresponding frame. Then, the relationship between approximation operators and modal formulas is as follows:
  \begin{itemize}
 \item[(i)] $[[\phi]]^{\lozenge}=[[\lozenge\phi]]$ and $[[\phi]]^{\square}=[[\square\phi]]$ for $\phi\in Fm(\textbf{RS})_{s_1}$.
\item[(ii)] $[[\phi]]^{\lozenge^{-1}}=[[\lozenge^{-1}\phi]]$ and $[[\phi]]^{{\square}^{-1}}=[[\square^{-1}\phi]]$ for  $\phi\in Fm(\textbf{RS})_{s_2}$.
  \end{itemize}}
\end{proposition}

\begin{definition}
    {\rm Let $\mathfrak{C}:=\{(G,M,I^{-1},I)\}$ be a frame based on the context $\mathbb{K}=(G,M,I)$. Then, we define
\begin{itemize}
\item[(a)] $Fm_{PC_{ext}}:=\{\phi\in Fm(\textbf{RS})_{s_1}\mid~ \models^\mathfrak{C}_{s_{1} } \square^{-1}\lozenge\phi\leftrightarrow\phi\}$ and $Fm_{PC_{int}}:=\{\phi\in  Fm(\textbf{RS})_{s_2}\mid~ \models^\mathfrak{C}_{s_{2} } \lozenge\square^{-1}\phi\leftrightarrow\phi\}$
\item[(b)] $Fm_{OC_{ext}}:=\{\phi\in  Fm(\textbf{RS})_{s_{1}}\mid ~ \models^\mathfrak{C}_{s_{1} } \lozenge^{-1}\square\phi\leftrightarrow\phi\}$ and $Fm_{OC_{int}}:=\{\phi\in  Fm(\textbf{RS})_{s_{2}}\mid ~ \models^\mathfrak{C}_{s_{2} }\square \lozenge^{-1}\phi\leftrightarrow\phi\}$
\item[(c)] $Fm_{PC}:=\{(\phi, \psi) \mid \phi\in Fm_{PC_{ext}}, \psi\in Fm_{PC_{int}}, \models^\mathfrak{C}_{s_{1}} \phi\leftrightarrow \square^{-1}\psi, \models^\mathfrak{C}_{s_{2}}  \lozenge\phi\leftrightarrow\psi\}$
 \item[(d)] $Fm_{OC}:=\{(\phi, \psi) \mid \phi\in Fm_{OC_{ext}}, \psi\in Fm_{OC_{int}},   \models^\mathfrak{C}_{s_{1}} \phi\leftrightarrow \lozenge^{-1}\psi,  \models^\mathfrak{C}_{s_{2}} \square\phi\leftrightarrow \psi\}$
    \end{itemize}}
\end{definition}

Obviously, when $(\phi, \psi)\in Fm_{PC}$, $([[\phi]], [[\psi]])\in{\mathcal PC}$ for any models based on ${\mathfrak C}$. Hence, $Fm_{PC}$ consists of pairs of formulas representing property oriented concepts. Analogously, $Fm_{OC}$ provides the representation of object oriented concepts. Note that these sets are implicitly parameterized by the underlying context and should be indexed with $\mathbb{K}$. However, for simplicity, we usually omit the index. 

\subsection{Two sorted modal logic and concept lattice in formal concept analysis}
\label{KF}
To represent formal concepts, we consider another two-sorted bidirectional signature $(\{s_{1}, s_{2}\}, \{\boxminus, \boxminus^{-1}\}\})$, where $\Sigma_{s_{1}s_{2}}=\{\boxminus\}$ and $\Sigma_{s_{2}s_{1}}=\{\boxminus^{-1}\}$, and the logic  \textbf{KF} based on it. Syntactically, the signature is the same as that for $\mathbf{KB}_{2}$ except we use different symbols to denote the modalities. Hence, formation rules of formulas remain unchanged and we denote the indexed family of formulas by $Fm(\textbf{KF})=\{Fm(\textbf{KF})_{s_{1}}, Fm(\textbf{KF})_{s_{2}}\}$.
In addition, while both \textbf{KF} and $\mathbf{KB}_{2}$ are interpreted in bidirectional models, the main difference between them is on the way of their modalities being interpreted.
\begin{definition}\label{windosatis}
{\rm Let $\mathfrak{M}:=(W_{1}, W_{2}, R, R^{-1}, v)$. Then, 
\begin{itemize}
     \item[(a)] For $\phi\in Fm(\textbf{KF})_{s_{1}}$ and $w\in W_2$, $\mathfrak{M},w \models_{s_{2}}\boxminus\phi$ iff for any $w'\in W_1$, $\mathfrak{M},w'\models_{s_{1}}\phi$ implies $R(w,w')$
    \item[(b)] For $\phi\in Fm(\textbf{KF})_{s_{2}}$ and $w\in W_1$, $\mathfrak{M},w \models_{s_{1}}\boxminus^{-1}\phi$ iff for any $w'\in W_2$, $\mathfrak{M},w'\models_{s_{2}}\phi$ implies $R^{-1}(w,w')$
\end{itemize}}
\end{definition}

The logic system $\mathbf{KF}:=\{\mathbf{KF}_{s_{1}}, \mathbf{KF}_{s_{2}} \}$ is shown in Figure~\ref{fig2}.
\begin{figure}[htp]\centering
	\framebox[110mm] {\parbox{100mm}{\begin{enumerate}
    \item Axioms:
\begin{itemize}
	\item[(PL)] Propositional tautologies of sort $s_i$ for $i=1,2$.
	\item [$(K_{\boxminus}^{1})$] $\boxminus(\phi_{1}\wedge\neg \phi_{2})\rightarrow (\boxminus\neg\phi_{1}\rightarrow \boxminus\neg\phi_{2})$ for $\phi_{1},\phi_{2}\in Fm(\mathbf{KF})_{s_1}$
    \item [$(B^{1})$]$ \phi\rightarrow \boxminus^{-1}\boxminus \phi$ for $\phi\in Fm(\mathbf{KF})_{s_1}$
     \item [$(K_{\boxminus^{-1}}^{2})$] $\boxminus^{-1}(\psi_{1}\wedge \neg\psi_{2})\rightarrow (\boxminus^{-1}\neg\psi_{1} \rightarrow \boxminus^{-1}\neg\psi_{2})$ for $\psi_{1},\psi_{2}\in Fm(\mathbf{KF})_{s_2}$
  \item [$(B^{2})$]  $\psi\rightarrow \boxminus\boxminus^{-1}\psi$ for $\psi\in Fm(\mathbf{KF})_{s_2}$
  \end{itemize}
\item Inference rules:
\begin{itemize}
    \item $(MP)_{s}$: for $s\in \{s_{1}, s_{2}\}$ and $\phi,\psi\in Fm_{s}$ 
    \[\infer{\psi}{\phi, \phi\rightarrow \psi}\]
    \item $(UG^{1}_{\boxminus})$: for $\phi\in Fm(\textbf{KF})_{s_{1}}$,
	 \[\infer{{\boxminus}\phi}{\neg\phi}\] 
  \item $(UG^{2}_{\boxminus^{-1}})$: for $\psi\in Fm(\textbf{KF})_{s_{2}}$, 
  \[\infer{{\boxminus}^{-1}\psi}{\neg\psi}\]
\end{itemize}
\end{enumerate}
}} \caption{The axiomatic system \textbf{KF}}\label{fig2}
\end{figure}

We define a translation $\rho:Fm(\textbf{KF})\rightarrow Fm(\mathbf{RS})$ where $\rho=\{\rho_{1}, \rho_{2}\}$ is defined  as follows:
\begin{enumerate}
\item $\rho_{i}(p):=p$ for all $p\in P_{s_{i}}$ for $i=1,2$.
\item $\rho_{i}(\phi\wedge\psi):=\rho_{i}(\phi)\wedge\rho_{i}(\psi)$ for $\phi, \psi\in Fm(\textbf{KF})_{s_i}$, $i=1,2$.
\item $\rho_{i}(\neg\phi):=\neg\rho_{i}(\phi)$ for $\phi\in Fm(\textbf{KF})_{s_i}$, $i=1,2$.
\item $\rho_{1}(\boxminus\phi):=\square\neg\rho_{1}(\phi)$ for $\phi\in Fm_{\textbf{KF}_{s_{1}}} $.
\item $\rho_{2}(\boxminus^{-1}\phi):=\square^{-1}\neg\rho_{2}(\phi)$ for $\phi\in Fm_{\textbf{KF}_{s_{2}}}$.
\end{enumerate}

\begin{theorem}
          {\rm
          \label{translation} For any formula $\phi\in Fm_{\textbf{KF}_{s{i}}}(i=1,2)$  the following hold.
          \begin{itemize}
              \item[(a)] $\Phi\vdash^{\textbf{KF}}\phi$ if and only if $\rho(\Phi)\vdash^{\mathbf{KB}_{2}}\rho(\phi)$ for any $\Phi\subseteq Fm_{\textbf{KF}_{s{i}}}$.
              \item[(b)]  Let $\mathfrak{M}:=(W_{1}, W_{2}, I, I^{-1}, v)$ be a model and $\mathfrak{M}^{c}:=(W_{1}, W_{2}, I^{c}, (I^{-1})^{c}, v)$ be the corresponding complemented model,  $w\in W_{i}, \mathfrak{M}, w\models_{s_{i}} \phi$ if and only if $\mathfrak{M}^{c}, w\models_{s_{i}} \rho(\phi)$ for all $i=1,2$.
              \item[(c)] $\phi$ is valid in the class $\mathcal{SF}_{2}$ if and only if $\rho(\phi)$ is valid in $\mathcal{SF}_{2}$.
          \end{itemize}
           }
\end{theorem}

\begin{proof}
\begin{itemize}
\item [(a).] We can prove it by showing that $\phi$ is an axiom  in \textbf{KF} if and only if $\rho(\phi)$ is an axiom in $\mathbf{KB}_{2}$, and for each rule in \textbf{KF}, there is a translation of it in $\mathbf{KB}_{2}$ and vice verse. 
\item [(b).] By induction on the complexity of formulas, as usual, the proof of basis and Boolean cases are straightforward.
For $\phi=\boxminus\psi$, let us assume any $w\in W_{2}$.  Then, by Definition \ref{windosatis}, $\mathfrak{M}, w\models_{s_{2}} \boxminus\phi$ iff for all $w'\in W_{1}$, $I^{c}ww'$ implies that $\mathfrak{M}, w'\models_{s_{1}}\neg\psi$. By induction hypothesis, this means that for all $w'\in W_{1}$, $I^{c}ww'$ implies that $\mathfrak{M}^c, w'\models_{s_{1}}\neg\rho(\psi)$. That is, $\mathfrak{M}^c, w\models_{s_{2}}\square\neg\rho(\phi)$. By definition of $\rho$, this is exactly  $\mathfrak{M^c}, w\models_{s_{2}} \rho(\phi)$. The case of $\phi=\boxminus^{-1}\psi$ is proved analogously. 
\item  [(c).] This follows immediately from (b).
\end{itemize}
\end{proof}

 \begin{proposition}\label{needproflattic}
{\rm \begin{itemize}
         \item[(a).] For $\phi_{1}, \phi_{2}\in Fm(\textbf{KF})_{s_{1}}$, $\infer{\boxminus\phi_{2}\rightarrow\boxminus\phi_{1}}{\phi_{1}\rightarrow\phi_{2}}$
        \item[(b).] For $\phi_{1}, \phi_{2}\in Fm(\textbf{KF})_{s_2}$, $\infer{\boxminus^{-1}\phi_{2}\rightarrow\boxminus^{-1}\phi_{1}}{\phi_{1}\rightarrow\phi_{2}}$
     \end{itemize}}
 \end{proposition}
\begin{proof}
    We only prove (a) and the proof of (b) is similar.
     \begin{align*}
     \vdash^{\textbf{KF}}& \phi_{1}\rightarrow\phi_{2}\\
     \vdash^{\mathbf{KB}_2}& \rho(\phi_{1})\rightarrow\rho(\phi_{2})
     (\mbox{\rm Theorem~\ref{translation} (a)})\\ 
        \vdash^{\mathbf{KB}_2}& (\rho(\phi_{1})\rightarrow\rho(\phi_{2}))\rightarrow(\neg\rho(\phi_{2})\rightarrow \neg\rho(\phi_{1}))(\mbox{\rm PL})\\
         \vdash^{\mathbf{KB}_2}& \neg\rho(\phi_{2})\rightarrow \neg\rho(\phi_{1})(\mbox{\rm MP})\\
         \vdash^{\mathbf{KB}_2}& \square(\neg\rho(\phi_{2})\rightarrow \neg\rho(\phi_{1}))(\mbox{\rm UG})\\
         \vdash^{\mathbf{KB}_2}& \square(\neg\rho(\phi_{2})\rightarrow \neg\rho(\phi_{1}))\rightarrow (\square\neg\rho(\phi_{2})\rightarrow\square\neg\rho(\phi_{1}))(\mbox{\rm K})\\
         \vdash^{\mathbf{KB}_2}&  \square\neg\rho(\phi_{2})\rightarrow{\square}\neg\rho(\phi_{1})(\mbox{\rm MP})\\
         \vdash^{\textbf{KF}}& \boxminus\phi_{2}\rightarrow\boxminus\phi_{1} (\mbox{\rm Theorem \ref{translation}(a)})
    \end{align*}
\end{proof}
\begin{theorem}
    {\rm \textbf{KF} is sound and strongly complete with respect to the class  $\mathcal{SF}_{2}$.}
\end{theorem}
\begin{proof}
    This follows from Theorem \ref{translation} and the fact that $
    \textbf{KB}_{2}$ is sound and strongly complete with respect to $\mathcal{SF}_{2}$.
\end{proof}

\begin{proposition}
  {\rm  Recalling the definition of truth set, we have
  \begin{itemize}
      \item[(i)] $[[\boxminus\phi]]=[[\phi]]^{+}$ for $\phi\in Fm(\textbf{KF})_{s_{1}}$
      \item [(ii)] $[[\boxminus^{-1}\phi]]=[[\phi]]^{-}$ for $\phi\in Fm(\textbf{KF})_{s_{2}}$.
  \end{itemize}}
\end{proposition}

\begin{definition}
    {\rm Let $\mathfrak{C}:=\{(G, M,I^{-1},I)\}$ be a frame based on the context $(G,M,I)$. Then, we define
    \begin{itemize}
        \item[(a)] $Fm_{FC_{ext}}:=\{\phi\in Fm(\mathbf{KF})_{s_1}\mid ~\models^\mathfrak{C}_{s_1} \boxminus^{-1}\boxminus\phi\leftrightarrow\phi\}$ and $Fm_{FC_{int}}:=\{\phi\in Fm(\mathbf{KF})_{s_{2}}\mid ~ \models^\mathfrak{C}_{s_{2} } \boxminus\boxminus^{-1}\phi\leftrightarrow\phi\}$

         \item[(b)] $Fm_{FC}:=\{(\phi, \psi) \mid \phi\in Fm_{PC_{ext}}, \psi\in Fm_{PC_{int}},  \models^\mathfrak{C}_{s_1} \phi\leftrightarrow \boxminus^{-1}\psi,  \models^\mathfrak{C}_{s_2}  \boxminus\phi\leftrightarrow \psi\}$
    \end{itemize}}
\end{definition}

In other words, the set $Fm_{FC}$ represents formal concepts induced from the context $(G, M, I)$.

\section{Logical representation of three concept lattices}
\label{lattice}
We have seen that a certain pairs of formulas in the logic \textbf{KF} and \textbf{KB}$_2$ can represent concepts in FCA and RST respectively.  The observation suggests the definition below.
\begin{definition}
          {\rm Let $\phi\in Fm(\textbf{RS})_{s_1}$, $\psi\in  Fm(\textbf{RS})_{s_2}$,  $\eta\in Fm(\textbf{KF})_{s_1}$, and  $\gamma\in  Fm(\textbf{KF})_{s_2}$. Then, for a context  $(G, M, I)$, we say that
          \begin{itemize}
              \item[(a)] $(\phi, \psi)$ is a  {\it (logical) property oriented concept} of $\mathbb{K}$ if $(\phi, \psi)\in Fm_{PC}$.
               \item[(b)] $(\phi, \psi)$ is a {\it  (logical) object oriented concept} of $\mathbb{K}$ if $(\phi, \psi)\in Fm_{OC}$.
                \item[(c)] $(\eta, \gamma)$ is a {\it (logical) formal concept} of $\mathbb{K}$ if $(\eta, \gamma)\in Fm_{FC}$.
          \end{itemize}}
      \end{definition}
We now explore the relationships between the three notions and their properties. In what follows, for a context $\mathbb{K}=(G,M,I)$, we usually use $\mathfrak{C}_{0}:=\{(G,M,I^{-1},I)\}$ and $\mathcal{C}_{1}:=\{(G,M, (I^{c})^{-1}),I^{c}\}$ to denote frames corresponding to $\mathbb{K}$ and $\mathbb{K}^c$ respectively.
\begin{proposition}\label{mapconcept}
{\rm  Let $\mathbb{K}:=(G, M, I)$ be a context. Then, 
      \begin{itemize}
          \item[(a)]   $(\phi, \psi)$  is a property oriented concept of $\mathbb{K}$ iff $(\neg\phi, \neg\psi)$ is an object oriented concept of $\mathbb{K}^{c}$ for $\phi\in Fm(\textbf{RS})_{s_1}$ and $\psi\in  Fm(\textbf{RS})_{s_2}$.
          \item[(b)] $(\phi, \psi)$  is a formal concept of $\mathbb{K}$  iff $(\rho(\phi), \neg \rho(\psi))$  is a property oriented concept of $\mathbb{K}^{c}$ for $\phi\in Fm(\textbf{KF})_{s_1}$ and $\psi\in  Fm(\textbf{KF})_{s_2}$.
          \item[(c)]   $(\phi,\psi)$  is a formal concept of $\mathbb{K}$  iff $(\neg\rho(\phi), \rho(\psi))$  is an object oriented concept of  $\mathbb{K}^{c}$ for $\phi\in Fm(\textbf{KF})_{s_1}$ and $\psi\in  Fm(\textbf{KF})_{s_2}$..
      \end{itemize}}
      \end{proposition}
      \begin{proof}
        \begin{itemize}
            \item [(a)] Suppose that $(\phi, \psi)$ is a property oriented concept of $\mathbb{K}$,  then by definition, $\models^{\mathfrak{C}_{0}}_{s_1}\square^{-1}\lozenge\phi\leftrightarrow\phi$, $ \models^{\mathfrak{C}_{0}}_{s_{2}}\lozenge\square^{-1}\psi\leftrightarrow\psi$, $ \models^{\mathfrak{C}_{0}}_{s_{1}} \phi\leftrightarrow \square^{-1}\psi$, and $ \models^{\mathfrak{C}_{0}}_{s_{2}} \lozenge \phi\leftrightarrow \psi$. Hence, we have the following derivation,
\begin{align*}
 &\models^{\mathfrak{C}_{0}}_{s_{1} } \square^{-1}\lozenge\phi\leftrightarrow\phi \\
 &\models^{\mathfrak{C}_{0}}_{s_{1} } (\square^{-1}\lozenge\phi\leftrightarrow\phi) \leftrightarrow (\neg\phi\leftrightarrow\neg\square^{-1}\lozenge\phi) \\
 &\models^{\mathfrak{C}_{0}}_{s_{1} }  \neg\phi\leftrightarrow\neg\square^{-1}\lozenge\phi \\
 &\models^{\mathfrak{C}_{0}}_{s_{1} }  \neg\phi\leftrightarrow\lozenge^{-1}\square\neg\phi \\
\end{align*}
Therefore, $\neg \phi\in Fm_{PC_{ext}}$.  Similarly, by $ \models^{\mathfrak{C}_{0}}_{s_{2} } \lozenge\square^{-1}\psi\leftrightarrow\psi$, contraposition and modus ponens, we can show that $\neg\psi\in Fm_{PC_{int}}$.

Using $ \models^{\mathfrak{C}_{0}}_{s_{1}} \phi\leftrightarrow \square^{-1}\psi $ , $\models^{\mathfrak{C}_{0}}_{s_{2}}  \lozenge \phi\leftrightarrow \psi$, contraposition and modus ponens , we can show that $(\neg\phi, \neg\psi)\in Fm_{OC}$.

We can also prove the converse direction by replacing  $\phi$,
$\psi$, $\lozenge$, $\square$ with $\neg\phi$, $\neg\psi$, $\square$, $\lozenge$ respectively.

\item [(b)] Because $(\phi,\psi)$ is a formal concept, we have $  \models^{\mathfrak{C}_{0}}_{s_1} \boxminus^{-1}\boxminus\phi\leftrightarrow\phi$, $\models^{\mathfrak{C}_{0}}_{s_2} \boxminus\boxminus^{-1}\psi\leftrightarrow\psi$, $\models^{\mathfrak{C}_{0}}_{s_1} \phi\leftrightarrow\boxminus^{-1}\psi$, and  $\models^{\mathfrak{C}_{0}}_{s_2}\boxminus \phi\leftrightarrow \psi$. By $\models^{\mathfrak{C}_{0}}_{s_1} \boxminus^{-1}\boxminus\phi\leftrightarrow\phi$ and Theorem \ref{translation} (a), we have $ \models^{\mathcal{C}_{1}}_{s_{1} } \rho(\boxminus^{-1}\boxminus\phi)\leftrightarrow\rho(\phi)$ which implies that $ \models^{\mathcal{C}_{1}}_{s_{1} }  \square^{-1}\lozenge\rho(\phi)\leftrightarrow\rho(\phi)$.  By $ \models^{\mathfrak{C}_{0}}_{s_{2} } \boxminus\boxminus^{-1}\psi\leftrightarrow\psi$ and Theorem \ref{translation} (a), we have $ \models^{\mathfrak{C}_1}_{s_{2} } \square\neg\square^{-1}\neg\rho(\psi)\leftrightarrow\rho(\psi)$, which implies that $ \models^{\mathfrak{C}_1}_{s_2} \lozenge\square^{-1}\neg\rho(\psi)\leftrightarrow\neg\rho(\psi)$.

Similarly, we can show that   $ \models^{\mathcal{C}_{1}}_{s_{1}} \rho(\phi)\leftrightarrow \square^{-1}\neg\rho(\psi) $ , $ \models^{\mathcal{C}_{1}}_{s_{2}}  \lozenge \rho(\phi)\leftrightarrow \neg\rho(\psi)$. Therefore, $(\rho(\phi), \rho(\psi))$ is a property oriented concept for $(G, M, I^{c})$. The proof for the converse direction is similar.
\item [(c)] It  follows from (a) and (b) immediately.
\end{itemize}
\end{proof}

Now, we can  define a relation $\equiv_{1}$  on the set  $Fm_{PC}$ as follows:
For $(\phi, \psi), (\phi', \psi')\in Fm_{PC}$, $(\phi, \psi) \equiv_{1}(\phi', \psi')$ if and only if $\models^{\mathfrak{C}_{0}} \phi\leftrightarrow\phi'$.

Analogously, we can define $\equiv_{2}$ and $\equiv_{3}$ on the set $Fm_{OC}$ and $Fm_{FC}$, respectively. Obviously, $\equiv_{1}, \equiv_{2}$ and $\equiv_{3}$ are all  equivalence relations. Let $Fm_{PC}/\equiv_{1}, Fm_{OC}/\equiv_{2}$, and $ Fm_{FC}/\equiv_{3}$ be the sets  of equivalence classes.

\begin{proposition}
\label{equivorder}
    {\rm For $(\phi, \psi), (\phi', \psi')\in Fm_{X}$, $(\phi, \psi) \equiv_i (\phi', \psi')$ iff $\models^{\mathfrak{C}_{0}} \psi\leftrightarrow\psi'$, where $i\in\{1,2,3\}$ for $X\in \{PC, OC, FC\}$ respectively.}
\end{proposition}
\begin{proof}
    Let us prove the case of $FC$ as an example. Suppose $(\phi, \psi), (\phi', \psi')\in Fm_{FC}$ and $(\phi, \psi)\equiv_{3}(\phi',\psi')$. Then, $\models^{\mathfrak{C}_0}_{s_1} \phi\leftrightarrow\phi'$, which implies $\models^{\mathfrak{C}_{0}}_{s_{2}}  \boxminus\phi\leftrightarrow\boxminus\phi'$ according to the semantics of \textbf{KF}. In addition, by definition of $Fm_{FC}$, $\models^{\mathfrak{C}_{0}}_{s_{2}}  \boxminus\phi\leftrightarrow\psi$, and $ \models^{\mathfrak{C}_{0}}_{s_{2}}  \boxminus\phi'\leftrightarrow\psi'$. Hence, $\models^{\mathfrak{C}_{0}}_{s_{2}}\psi\leftrightarrow\psi'$.

    Proofs for other two cases are similar.
\end{proof}

\begin{proposition}
\label{intentextfrom}
    {\rm Let $X\in \{PC_{ext}, OC_{ext}, FC_{ext}\}$  and $Y\in \{PC_{int}, OC_{int}, FC_{int}\}$. Then,
\begin{itemize}
        \item [(a)] $Fm_{X}$ and $Fm_{Y}$ are closed under conjunction.
        \item [(b)] If $\phi\in Fm_{X}$ and $\psi\in Fm_{Y}$, then $\circ\phi\in Fm_{Y}$ and $\circ^{-1}\psi\in Fm_{X}$, where $\circ\in\{\boxminus, \square, \lozenge\}$ depending on $X$ and $Y$ according to their respective definitions.
\end{itemize}.}
\end{proposition}
\vspace{-1cm}
\begin{proof}
\begin{itemize}
\item [(a).] We prove the case of $Fm_{FC_{ext}}$ as an example and other cases can be proved in a similar way. Let $\phi, \phi'\in Fm_{FC_{ext}}$. Then, $ \models^{\mathfrak{C}_{0}}_{s_{1}} \boxminus^{-1}\boxminus\phi\leftrightarrow\phi$ and $\models^{\mathfrak{C}_{0}}_{s_{1}} \boxminus^{-1}\boxminus\phi'\leftrightarrow\phi'$. By using the translation $\rho$ and Theorem~\ref{translation}, we have both 
$\models^{\mathfrak{C}_{0}}_{s_{1}}\boxminus^{-1}\boxminus(\phi\wedge\phi')\rightarrow\boxminus^{-1}\boxminus\phi$ and $\models^{\mathfrak{C}_{0}}_{s_{1}}\boxminus^{-1}\boxminus(\phi\wedge\phi')\rightarrow\boxminus^{-1}\boxminus\phi'$. Hence, we can derive $\models^{\mathfrak{C}_{0}}_{s_{1}}\boxminus^{-1}\boxminus(\phi\wedge\phi')\rightarrow(\phi\wedge\phi')$. Also, with the translation, we have     $\models^{\mathfrak{C}_{0}}_{s_{1}}(\phi\wedge\phi')\rightarrow  \boxminus^{-1}\boxminus(\phi\wedge\phi')$ because the formula is mapped to an instance of axiom (B).  Hence, $\phi\wedge\phi'\in Fm_{FC_{ext}}$.
\item [(b).] Let us prove the case of $\phi\in Fm_{FC_{ext}}$ as an example.  Assume that $\phi\in Fm_{FC_{ext}}$ and $\circ=\boxminus$. Then, according to the semantics of $\boxminus$, $\models^{\mathfrak{C}_{0}}_{s_{1}}\boxminus^{-1}\boxminus\phi\leftrightarrow\phi$ implies $\models^{\mathfrak{C}_{0}}_{s_{1}}\boxminus\boxminus^{-1}\boxminus\phi\leftrightarrow\boxminus\phi$ . Hence $\boxminus\phi\in Fm_{FC_{int}}$. Similarly, if  $\psi\in Fm_{FC_{int}}$, then $\boxminus^{-1}\psi\in Fm_{FC_{ext}}$.
\end{itemize}
\end{proof}

From the proposition, we can derive the following corollary immediately.
\begin{corollary}\label{concept}
    {\rm
    \begin{itemize}
        \item[(a)] $(\phi_{1}, \psi_{1}), (\phi_{2}, \psi_{2})\in Fm_{PC}$ implies that $(\phi_{1}\wedge\phi_{2}, \lozenge(\phi_{1}\wedge\phi_{2}))$  and $ (\square^{-1}(\psi_{1}\wedge\psi_{2}), \psi_{1}\wedge\psi_{2} )\in Fm_{PC}$.
        \item[(b)] $(\phi_{1}, \psi_{1}), (\phi_{2}, \psi_{2})\in Fm_{OC}$ implies that $(\phi_{1}\wedge\phi_{2}, \square(\phi_{1}\wedge\phi_{2}))$ and $ (\lozenge^{-1}(\psi_{1}\wedge\psi_{2}), \psi_{1}\wedge\psi_{2} )\in Fm_{OC}$.
        \item[(c)] $(\phi_{1}, \psi_{1}), (\phi_{2}, \psi_{2})\in Fm_{FC}$ implies that $(\phi_{1}\wedge\phi_{2}, \boxminus(\phi_{1}\wedge\phi_{2}))$ and $ (\boxminus^{-1}(\psi_{1}\wedge\psi_{2}), \psi_{1}\wedge\psi_{2} )\in Fm_{FC}$.
    \end{itemize}}
\end{corollary}
Now we can define the following structures:

$(Fm_{PC}/\equiv_{1}, \vee_{1}, \wedge_{1})$, where $[(\phi, \psi)],  [(\phi', \psi')]\in Fm_{PC}/\equiv_{1}$,
\vspace{-.1cm}
 \begin{align*}
            [(\phi, \psi)]\wedge_{1}[(\phi', \psi')]&:= [(\phi\wedge\phi', \lozenge(\phi\wedge\phi') )]\\
            [(\phi, \psi)]\vee_{1}[(\phi', \psi')]&:= [(\square^{-1}(\psi\wedge\psi'), (\psi\wedge\psi') )]
        \end{align*}

 $(Fm_{OC}/\equiv_{2}, \vee_{2},  \wedge_{2})$, where
        $[(\phi, \psi)],  [(\phi', \psi')]\in Fm_{OC}/\equiv_{2}$,

        \vspace{-.4cm}
        \begin{align*}
            [(\phi, \psi)]\wedge_{2}[(\phi', \psi')]&:= [(\phi\wedge\phi', \square(\phi\wedge\phi') )]\\
            [(\phi, \psi)]\vee_{2}[(\phi', \psi')]&:= [(\lozenge^{-1}(\psi\wedge\psi'), (\psi\wedge\psi') )]
        \end{align*}

 $(Fm_{FC}/\equiv_{3}, \vee_{3}, \wedge_{3})$, where  $[(\phi, \psi)], [(\phi', \psi')]\in Fm_{FC}/\equiv_{3}$,
\vspace{-.1cm}
 \begin{align*}
     [(\phi, \psi)]\wedge_{3}[(\phi', \psi')]&:=[(\phi\wedge\phi', \boxminus(\phi\wedge\phi') )]\\
     [(\phi, \psi)]\vee_{3}[(\phi', \psi')]&:=[(\boxminus^{-1}(\psi\wedge\psi'), (\psi\wedge\psi') )]
 \end{align*}

       \begin{theorem}
          {\rm  For a context $\mathbb{K}$, $(Fm_{PC}/\equiv_{1}, \vee_{1}, \wedge_{1})$, $(Fm_{OC}/\equiv_{2}, \vee_{2}, \wedge_{2})$ and $(Fm_{FC}/\equiv_{3}, \vee_{3}, \wedge_{3})$,  are  lattices.
          }
       \end{theorem}
       \begin{proof}
           We give proof for the structure $(Fm_{FC}/\equiv_{3}, \vee_{3}, \wedge_{3})$ and the proofs of other cases are similar. Let $(\phi, \psi), (\phi_{1}, \psi_{1}), (\phi^{\prime}, \psi^{\prime}), (\phi^{\prime}_{1}, \psi^{\prime}_{1})\in Fm_{FC}$ such that $(\phi, \psi)\equiv_{3}(\phi_{1}, \psi_{1})$ and $(\phi^{\prime}, \psi^{\prime})\equiv_{3}(\phi^{\prime}_{1},\psi^{\prime}_{1})$. By Corollary \ref{concept}, $(\phi\wedge\phi^{\prime}, \boxminus(\phi\wedge\phi^{\prime})), (\boxminus^{-1}(\psi\wedge\psi^{\prime}), \psi\wedge\psi^{\prime} ), (\phi_{1}\wedge\phi_{1}^{\prime}, \boxminus(\phi_{1}\wedge\phi_{1}^{\prime}))$ and $ (\boxminus^{-1}(\psi_{1}\wedge\psi_{1}^{\prime}), \psi_{1}\wedge\psi_{1}^{\prime} )\in Fm_{FC}$.  Now $(\phi, \psi)\equiv_{3}(\phi_{1}, \psi_{1})$ and $(\phi^{\prime}, \psi^{\prime})\equiv_{3}(\phi^{\prime}_{1},\psi^{\prime}_{1})$ implies that $\models^{\mathfrak{C}_{0}} \phi\leftrightarrow\phi_{1} $ and $\models^{\mathfrak{C}_{0}}\phi^{\prime}\leftrightarrow\phi_{1}^{\prime}$. By Proposition \ref{equivorder},  $\models^{\mathfrak{C}_{0}} \psi\leftrightarrow\psi_{1} $ and $\models^{\mathfrak{C}_{0}}\psi^{\prime}\leftrightarrow\psi_{1}^{\prime}$. $\models^{\mathfrak{C}_{0}} \phi\wedge\phi^{\prime}\leftrightarrow\phi_{1}\wedge\phi_{1}^{\prime}$ and  $\models^{\mathfrak{C}_{0}} \psi\wedge\psi^{\prime}\leftrightarrow\psi_{1}\wedge\psi_{1}^{\prime}$ which implies that  $(\phi\wedge\phi^{\prime}, \boxminus(\phi\wedge\phi^{\prime}))\equiv_{3} (\phi_{1}\wedge\phi_{1}^{\prime}, \boxminus(\phi_{1}\wedge\phi_{1}^{\prime}))$ and  $(\boxminus^{-1}(\psi\wedge\psi^{\prime}), \psi\wedge\psi^{\prime} )\equiv_{3} (\boxminus^{-1}(\psi_{1}\wedge\psi_{1}^{\prime}), \psi_{1}\wedge\psi_{1}^{\prime} )$. Hence,  $\wedge_{3}$ and $\vee_{3}$  are well-defined operations. Their commutativity and associativity follow from the fact that $\vdash^{\textbf{KF}}\phi\wedge\psi\leftrightarrow\psi\wedge\phi$ and  $\vdash^{\textbf{KF}} (\phi\wedge\psi)\wedge\gamma\leftrightarrow \phi\wedge(\psi\wedge\gamma)$. Now we will show that for all $[(\phi_{1}, \psi_{1})], [(\phi_{2}, \psi_{2})]\in Fm_{FC}/\equiv_{3}$, $[(\phi_{1}, \psi_{1})]\wedge ([\phi_{1}, \psi_{1}]\vee [(\phi_{2}, \psi_{2})])=[(\phi_{1}, \psi_{1})]$ which is equivalent to  $[(\phi_{1}\wedge\boxminus^{-1}(\psi_{1}\wedge\psi_{2}), \boxminus(\phi_{1}\wedge\boxminus^{-1}(\psi_{1}\wedge\psi_{1})))]=[(\phi_{1}, \psi_{1})]$. We know that $\models^{\mathfrak{C}_{0}} \phi_{1}\wedge \boxminus^{-1}(\psi_{1}\wedge\psi_{2})\rightarrow\phi_{1}$. In addition,
\begin{align*}
    & \models^{\mathfrak{C}_{0}}_{s_{1}} \phi_{1}\leftrightarrow \boxminus^{-1}\psi_{1} ~\mbox{as}~ (\phi_{1}, \psi_{1})\in Fm_{FC}\\
&\models^{\mathfrak{C}_{0}}_{s_{2}}\psi_{1}\wedge \psi_{2}\rightarrow \psi_{1}\\
&\models^{\mathfrak{C}_{0}}_{s_{1}}\boxminus^{-1}\psi_{1}\rightarrow \boxminus^{-1}(\psi_{1}\wedge \psi_{2})~\mbox{by Proposition \ref{needproflattic}}\\
     & \models^{\mathfrak{C}_{0}}_{s_{1}} \phi_{1}\rightarrow \boxminus^{-1}(\psi_{1}\wedge \psi_{2})\\
      & \models^{\mathfrak{C}_{0}}_{s_{1}} \phi_{1}\rightarrow\phi_{1}\wedge \boxminus^{-1}(\psi_{1}\wedge \psi_{2})
    \end{align*}
           So $ \models^{\mathfrak{C}_{0}}_{s_{1}} \phi_{1}\leftrightarrow \boxminus^{-1}(\psi_{1}\wedge \psi_{2})$ which implies that $[(\phi_{1}, \psi_{1})]\wedge ([\phi_{1}, \psi_{1}]\vee [(\phi_{2}, \psi_{2})])=[(\phi_{1}, \psi_{1})]$. Analogously, we can show that $[(\phi_{1}, \psi_{1})]\vee ([\phi_{1}, \psi_{1}]\wedge [(\phi_{2}, \psi_{2})])=[(\phi_{1}, \psi_{1})]$. Hence   $(Fm_{FC}/\equiv_{3}, \vee_{3}, \wedge_{3})$ is a lattice.
       \end{proof}

       \begin{theorem}
           {\rm Let $\mathbb{K}$ be a context and let $\mathbb{K}^{c}$ be its corresponding complemented context. Let $Fm_{FC}$ be the set of logical formal concepts of $\mathbb{K}$ and let $Fm_{PC}$ and $Fm_{OC}$ be the sets of logical property oriented concepts and logical object oriented concepts of $\mathbb{K}^{c}$ respectively. Then, \begin{itemize}
\item[(a)] $(Fm_{FC}/\equiv_{3}, \vee_{3}, \wedge_{3})$ and  $(Fm_{PC}/\equiv_{1}, \vee_{1}, \wedge_{1})$ are isomorphic.
\item[(b)] $(Fm_{PC}/\equiv_{1}, \vee_{1}, \wedge_{1})$ and $(Fm_{OC}/\equiv_{2}, \vee_{2}, \wedge_{2})$  are dually isomorphic.
\item[(c)] $(Fm_{FC}/\equiv_{3}, \vee_{3}, \wedge_{3})$ and $(Fm_{OC}/\equiv_{2}, \vee_{2}, \wedge_{2})$ are dually isomorphic.
\end{itemize}}
       \end{theorem}
       \begin{proof}
           (a)  By Proposition \ref{mapconcept}, the mapping  $h:Fm_{FC}/\equiv_{3}\rightarrow Fm_{PC}/\equiv_{1} $ defined by $h([(\phi, \psi)]):=[(\rho(\phi), \rho(\neg\psi))]$ is well-defined and surjective. Now $h([(\phi_{1}, \psi_{1})])=h([(\phi_{2}, \psi_{2})])$ implies  $[\rho((\phi_{1}), \rho(\neg\psi_{1}))]=[(\rho(\phi_{2}), \rho(\neg\psi_{2}))]$, which in turn implies $\models^{\mathfrak{C}_1} \rho(\phi_{1})\leftrightarrow\rho(\phi_{2})$, and by Theorem \ref{translation}, $\models^{\mathfrak{C}_0} \phi_{1}\leftrightarrow\phi_{2}$. This means that 
         $[(\phi_{1}, \psi_{1})]=[(\phi_{2}, \psi_{2})]$. Thus, $h$ is injcetive, and as a result, $h$ is a bijection. In addition,
           \begin{align*}
               h([(\phi_{1}, \psi_{1})]\wedge_3 [(\phi_{2}, \psi_{2})])&= h([(\phi_{1}\wedge \phi_{2}, \boxminus(\phi_{1}\wedge\phi_{2}))])\\
               &=([\rho(\phi_{1}\wedge\phi_{2}),\rho(\neg\boxminus(\phi_{1}\wedge\phi_{2}))])\\
                &=([\rho(\phi_{1}\wedge\phi_{2}),\lozenge\rho(\phi_{1}\wedge\phi_{2})])\\
                &= h([(\phi_{1}, \psi_{1})])\wedge_1 h([(\phi_{2}, \psi_{2})])
           \end{align*}
           Therefore, $h$ is an isomorphism.

           \noindent (b) Analogously, we can show that $f: Fm_{PC}/\equiv_{1}\rightarrow Fm_{OC}/\equiv_{1} $ such that $f([(\phi, \psi)]):=[(\neg\phi, \neg\psi)]$ is a dual isomorphism.

           \noindent (c) It follows from (a) and (b) immediately.
       \end{proof}
\section{Conclusion and future direction}
\label{conclusion}
In this paper, we show that concepts based on RST and FCA can be represented in two dual instances of two-sorted modal logics \textbf{KB} and \textbf{KF}. An interesting question is how to deal with both kinds of concepts in a single framework. To address the question, we apparently need a signature including all modalities in \textbf{KB} and \textbf{KF} together. For that, the Boolean modal logic proposed in  \cite{Gargov1987} may be helpful. Hence, to investigate many-sorted Boolean modal logic and its representational power for concepts based on both RST and FCA will be an important direction in our future work.

As a formal context consists of objects, properties, and a relation  between them, the relationship between objects and properties can change over time. Hence, to model and analyze the dynamics of contexts is also desirable. Using two-sorted bidirectional relational frames, we can model contexts at some time. Therefore, integrating temporal logic with many-sorted modal logic will provide an approach to model dynamics of contexts. This is another possible direction for further research.

\end{document}